\documentclass{article}

\usepackage{amsmath}
\usepackage{amssymb}
\usepackage{amsthm}
\usepackage{enumerate}
\usepackage{bbm}

\usepackage{multicol}
\usepackage{amsfonts}
\usepackage[colorlinks=true,citecolor=red,linkcolor=blue]{hyperref}

\newtheorem{theorem}{Theorem}[section]
\newtheorem{proposition}[theorem]{Proposition}
\newtheorem{corollary}[theorem]{Corollary}
\newtheorem{lemma}[theorem]{Lemma}
\newtheorem{problem}[theorem]{Problem}
\newtheorem{rmk}[theorem]{Remark}

\newtheorem*{theorem*}{Theorem}
\newtheorem*{corollary*}{Corollary}

\newtheorem{definition}[theorem]{Definition}

\newcommand{\pp}{\mathfrak{p}}

\renewcommand{\O}{\mathcal {O}}
\newcommand{\F}{\mathbb{F}}
\newcommand{\MM}{\mathcal{M}}
\newcommand{\LL}{\mathcal{L}}
\newcommand{\ord}{\operatorname{ord}}

\DeclareMathOperator{\ass}{as}
\DeclareMathOperator{\dib}{\operatorname{|}}
\DeclareMathOperator{\Div}{\operatorname{div}}

\newcommand{\SQ}{\operatorname{SQ}}
\newcommand{\Sq}{\operatorname{Sq}}

\newcommand{\qq}{\mathfrak{q}}

\title{Existential decidability for addition and divisibility in holomorphy subrings of global fields}
\author{Carlos  Mart\'inez-Ranero, Javier Utreras, Xavier Vidaux\footnote{Universidad de Concepci\'on, Chile (all three authors).}}

\begin{document}
\maketitle

\begin{abstract}
We investigate the problem of deciding whether a system of linear equations, together with divisibility conditions on the variables, has a solution over holomorphy subrings of global fields. We obtain decidability results when we allow poles at a cofinite set of places, and undecidability results when at a finite set of places. 
%
%
%
\footnote{ The  first named author was partially supported by  Proyecto VRID-Investigación  No. 220.015.024-INV. The second author has been supported by the Fondecyt postdoctorado project 3160301 from Conicyt, Chile. The third author has been partially supported by the Fondecyt project 1170315 from Conicyt, Chile.}
\end{abstract}



\tableofcontents

\section{Introduction and main results}

In the late seventies, Lipshitz \cite{L77}, and in parallel Bel’tyukov \cite{B76}, showed that there is an algorithm for the following decision problem:\\
 Given any sentence of the form 
\begin{equation} \label{eqn:intro}
\exists x_1,\dots, \exists x_n
\bigwedge_{i=1}^k f_i(x_1,\dots,x_n)=0\wedge
\bigwedge_{i=1}^m g_i(x_1,\dots,x_n) \ |\  h_i(x_1,\dots,x_n),
\end{equation}

\noindent where the $f_i$, $g_i$ and $h_i$ are linear polynomials with integer coefficients, decide whether or not  it has a solution over $\mathbb{Z}$.

In contrast, Lipshitz \cite{L78b} showed that the analogous problem for the ring of integers $\mathcal{O}_K$ of any number field other than quadratic imaginary, is as complicated as Hilbert Tenth Problem for $\mathcal{O}_K$ (i.e. considering polynomials of arbitrary degree in the input). The latter is widely believed to be unsolvable (for unconditional results, see \cite{Den75,DL78,P88,S89,V89,SS89} and more recently \cite{GFP20}; conditional results have been obtained for every $K$ --- \cite{Po02,CPhZ05,Sh08,MR10,MR18}).

Complementing the above result of Lipshitz and Beltyukov, it was recently shown by Cerda-Romero and Mart\'inez-Ranero \cite{CM17} that if $S$ is a finite and non-empty set of primes, then there is no algorithm to decide whether Equation \eqref{eqn:intro} has solutions in $\mathbb Z[S^{-1}]$. There are also similar results for rings of rational functions over finite fields see  \cite{P85},\cite{P88} and \cite{CM}. 

In this paper, we will extend the above results to holomorphy subrings of global fields. In order to make this precise we need to introduce a few definitions. 

Let $K$ be a global field with places at infinity $\infty_1,\dots,\infty_s$, and let $\O_{K}$ be the ring  of elements integral at all places other than $ \infty_1,\dots,\infty_s$. In the case where $K$ is a number field, we will be using the term places at infinity in its conventional meaning to denote the archimedean valuations, and in the case of function fields (i.e.: $K$ is a finite separable extension of $F_p(t)$ for some rational prime $p$), we will let $\infty_1,\dots,\infty_s$ denote the primes of $K$ lying above $\infty$. In the latter case, the choice of places at infinity is somewhat arbitrary but it is closely analogous to the number field case.  For any set $S$ of primes of $K$ containing all infinite primes we define $\O_{K,S}$ to be the holomorphy ring $$\O_{K,S}:=\{\, x\in K\colon \ord_\pp(x)\geq 0, \emph{for all}\ \pp\notin S\,\}.$$

\bigskip
This paper is motivated by the following general problem.  \begin{problem}\label{prob}
	For which holomorphy rings $\O_{K,S}$ does there exist an algorithm to solve the following decision problem: for any given system of linear equations with coefficients in $\mathbb Z$ (or in $\mathbb{F}_p$ or $\mathbb{F}_p[t]$ for the function field case) together with divisibility conditions on the variables, decide whether or not there exists a solution over $\O_{K,S}$. 

\end{problem}

In other words, in terms of mathematical logic, the problem asks which of the rings $\O_{K,S}$ have a decidable positive existential theory over the languages $\mathcal{L}_{\Div}=\{=,0,1,+,\mid\}$ and $ \mathcal{L}_{\Div,t}=\mathcal{L}_{\Div}\cup \{\cdot t\}$ --- here $x\mid y$ will be interpreted as ``$x$ divides $y$'', and $\cdot t$ as multiplication by $t$. \\

In this paper, we provide an answer to Problem \ref{prob} for two different cases. First we obtain decibility results for the case where $S$ is a cofinite set of places.

\begin{theorem}[Global fields over $\mathcal{L}_{\rm div}$]\label{maindec1}
Let $K$ be a number field or a finite separable extension of $\mathbb{F}_p(t)$, and let $\O_K$ be its ring of integers.  If $S$ is a cofinite set of prime ideals of $\O_K$, then the existential theory of $\O_{K,S}$ is decidable in the language $\mathcal{L}_{\rm div}$.
\end{theorem}

\begin{theorem}[Function fields over $\mathcal{L}_{\rm div,t}$]\label{maindec2}
Assume that the existential theory of $\mathbb F_p((t))$ over the language $\mathcal L_{\rm ring,\rm ord}=\{0,1,+,\cdot,=,\cdot t,\ord\}$ is decidable (where $\ord$ is interpreted as the valuation ring, and $\cdot t$ is multiplication by $t$). 

Let $K$ be a finite separable extension of $\mathbb{F}_p(t)$, and let $\O_K$ denote the integral closure of $\mathbb F_p[t]$ in $K$. If $S$ is a cofinite set of prime ideals of $\O_K$, then the existential theory of $\O_{K,S}$ is decidable in the language $\mathcal L_{\rm div, \rm t}$.
\end{theorem}

It is not known whether the existential theory of $\mathbb F_p((t))$ is decidable as a valued field. By a result of Denef and Schoutens \cite{DS99}, it would follow from resolution of singularities in positive characteristic --- we thank Sylvy Anscombe and Arno Fehm for pointing out to us that the analogue of Theorem \ref{maindec2} without $\cdot t$ in the language can be obtained unconditionally.

The main idea of the proof is to translate any given formula to a corresponding formula in the completion at each prime in the complement of $S$, and then use some tools from approximation theory to relate these individual decision problems to the original one. 



In contrast we obtain an undecidability result for the case where $S$ is a finite set of places of cardinality larger than one.

\begin{theorem}\label{mainund}
Assume that $K$ has odd characteristic. If $S$ is finite with at least two elements, then multiplication is positive-existentially definable over the structure $(\O_{K,S}, 0,1, +, |, \cdot t )$. In particular, the positive-exitential theory of this structure is undecidable.
\end{theorem}

The above Theorem  extends the results of Lipshitz \cite{L78b} to the ring of $S$-integers\footnote{When $S$ is finite the holomorphy rings are usually called rings of $S$-integers.} of a global function field. 
The idea of the proof is roughly as follows. We break the proof into three steps. First we show that being different from zero is positive-existentially definable. This step turns out to be quite technical and requires a deep Theorem of Lenstra \cite{Len77} from analytic number theory. The second step is to define the square function restricted to the set of units. This is the easiest part and just uses basic facts of valuation theory. The third step consists of positive-existentially defining the square function on $\O_{K,S}$. This is the hardest part. The formulas we use are quite similar to the ones used in \cite{L78b} and \cite{CM17}. However, the proofs are quite different, the main new ingredients are the use of the pigeon hole principle (which allows us to improve the result of Lipshitz) and the use of the  ideal class number and the lattice of units of $\O_{K,S}^\times$ (which allows us to extend the results of \cite{CM17}).

The structure of the paper is as follows. Section \ref{decres} details the proof of Theorems \ref{maindec1} and \ref{maindec2}. The undecidability result in Theorem \ref{mainund} is covered in Section \ref{undres}. It is worth mentioning that both sections can be read independently.

%
\section{Decidability results}\label{decres}

The purpose of this section is to prove Theorems \ref{maindec1} and \ref{maindec2}.
Let $R$ be either $\mathbb Z$ or $\mathbb F_p[t]$, $F$ the fraction field of $R$, $K/F$ a finite separable extension and  $\O_K$  the ring of integers of $K$. We consider the languages $\LL_{\Div}=\{0,1,+,|,= \}$ and $\LL_{\Div,t}=\LL_{\Div}\cup \{\cdot t\}$. 
Let $T=\{\mathfrak{p}_1,\dots ,\mathfrak{p}_n\}$ be a finite set of prime ideals of $\O_K$ not containing any of the primes at infinity, and let $S$ be its complement in the set of prime ideals. We consider the holomorphy ring $\O_{K,S}$ as an $\mathcal{L}_{\Div}$-structure ($\LL_{\Div,t}$-structure) by interpreting each symbol as expected. We will refer to this structure as $\mathcal{M}_S$.

For ease of reading, we will summarize the steps of the decision algorithm. 

\noindent\textbf{Problem.} Determine whether a system of equations and relations of the form
$$
\varphi=\bigwedge\varphi_i(\overline x)
$$
in the given language has a solution in $\O_{K,T}$.

\begin{enumerate}[Step 1.]
	\item Construct an equivalent system
	$$
	\overline\varphi=\bigwedge\overline\varphi_i(\overline x)
	$$
	in a more convenient language (referred to as $\overline\LL$, cf. Proposition \ref{prop:Lbar}).
	
	\item Produce the system $L_\varphi$ of all linear equations among the $\overline\varphi_i$.
	
	\item Use Gaussian Elimination to decide whether the system $L_\varphi$ has a solution in $F$.
	\begin{itemize}
		\item If it does not have a solution, the answer to the problem is \textbf{negative} (cf. Lemma \ref{lemma:linear}).
		
		\item If it has a solution, its solution set is definable and obtainable from Gaussian Elimination.
	\end{itemize}
	
	\item For each prime ideal $\pp$ in $T$, and using the solution set obtained in Step 3, produce a specific sentence $\varphi_\pp$ in the language of valued fields, as described in Section \ref{sect:translate}.
	
	\item For each prime ideal $\pp$ in $T$, decide whether $\varphi_\pp$ holds in the respective completion $K_\pp$ (doable -- cf. Proposition \ref{prop:completions}).
	\begin{itemize}
		\item If each one of them holds, the answer to the problem is \textbf{positive} (cf. Theorem \ref{thm:main}).
		
		\item If one of them does not hold, the answer to the problem is \textbf{negative} (cf. Proposition \ref{prop:local}).
	\end{itemize}
\end{enumerate}

\subsection{Decidability of the completions}

For this subsection, fix a prime ideal $\pp\in T$. Let $\bf{K}_\pp$ be the completion of $K$ with respect to the $\pp$-valuation. Also let $\LL_{\ord}=\{0,1,+,\cdot,=,\ord\}$ be the language of valued fields, and $\LL_{\ord,t}=\LL_{\ord}\cup\{t\}$, where $t$ is a symbol of constant interpreted by $t$. The aim of this section is to prove the following proposition, which may be known. We provide some proofs because we were unable to find them in the literature.
\begin{proposition} \label{prop:completions}
\begin{itemize}
\item If $F=\mathbb{Q}$, then the first--order theory of $\bf{K}_\pp$ in the language $\LL_{\ord}$ is decidable.
\item \cite{AF16} If $F=\mathbb{F}_p(t)$, then the existential theory of $\bf{K}_\pp$ in the language $\LL_{\ord}$ is decidable.
\item Assume that the existential theory of $\mathbb F_p((t))$ in the language $\LL_{\ord}$ is decidable. If $F=\mathbb{F}_p(t)$, then the existential theory of $\bf{K}_\pp$ in the language $\LL_{\ord,t}$ is decidable.
\end{itemize}
\end{proposition}

\begin{proof}
By \cite[Chap. II, Cor. 8.4]{Neuk99}, $\bf{K}_\pp$ is a finite extension of ${\bf F}_P$, where $P$ is the prime in $F$ lying below $\pp$. If the degree of this extension is $n$, we want to interpret $\bf{K}_\pp$ as ${\bf F}_P^n$ in the ring structure of ${\bf F}_P$.

Write ${\bf K}_\pp={\bf F}_P(\alpha)$. The field ${\bf K}_\pp$ can be interpreted as ${\bf F}_P^n$ with componentwise addition, and a multiplication map defined using the coefficients of the minimal polynomial of $\alpha$. By Krasner's Lemma (cf. \cite[3.4.2, Prop. 5]{BGR84}) and the separability of the extension ${\bf K}_\pp / {\bf F}_P$, we can choose $\alpha$ such that these coefficients lie in $F$, and hence this interpretation is definable without quantifiers.

The valuation ring of ${\bf K}_\pp$ is existentially definable in the ring structure  ${\bf K}_\pp$ (see \cite{AK14} or \cite{Fehm15}), and hence existentially definable in the ring structure  ${\bf F}_P$.

The conclusions follow from the corresponding results by Ersh\"ov \cite{E65} and Ax-Kochen \cite{AK65} (for $\mathbb{Q}_p$) and our main assumption (for $\mathbb{F}_p((t))$).
\end{proof}

\subsection{Reducing divisibility}

Let $\overline{\mathcal{L}}=(\mathcal{L}_{\Div}\setminus\{|\})\cup\{\neq,\|,o_1,\dots ,o_n \}$ (adding $\cdot t$ for the proof of Theorem \ref{maindec2} --- the proofs for Theorems \ref{maindec1} and \ref{maindec2} are identical save for this fact), where all of these new symbols are binary relation symbols. We interpret $\neq$ as inequality, $\|$ as the binary relation given by
\[
x\|y\textrm{ if and only if }x|y\textrm{ and }y\neq 0,
\]
and each $o_i$ as the relation
\[
o_i(x,y)\textrm{ if and only if }x\neq 0\textrm{ and }\ord_{\pp_i}(x)>\ord_{\pp_i}(y).
\]


\begin{proposition} \label{prop:Lbar}
For each existential $\LL$-formula $\varphi(\overline x)$ there exists a positive-existential $\overline{\LL}$-formula $\overline{\varphi}(\overline x)$ such that, for every $\overline c\in \O_{K,S}$,
\[
\MM_S\models \varphi(\overline c)\textrm{ if and only if }\overline{\MM}_S\models \overline\varphi(\overline c).
\]
Moreover, $\overline{\varphi}(\overline x)$ can be effectively constructed from $\varphi(\overline x)$.
\end{proposition}

\begin{proof}
Note that, for $x,y\in \O_{K,S}$,
\[
\MM_S\models x|y\textrm{ if and only if }\overline\MM_S\models y=0\vee x\| y
\]
and
\[
\MM_S\models\neg x|y\textrm{ if and only if }\overline\MM_S\models (x=0\wedge y\neq 0)\vee\bigvee_{i=1}^n o_i(x,y).
\]
\end{proof}

\begin{corollary} \label{corollary:normal-form}
Every existential $\LL$-sentence is equivalent (in the sense of the previous Proposition) to a disjunction of $\overline\LL$-sentences of the form
\[
\exists\overline x\bigwedge \varphi_i(\overline x)
\]
where each $\varphi_i(\overline x)$ is of one of the forms
\begin{enumerate}[(i)]
\item $l_1(\overline x)\|l_2(\overline x)$,
\item $l_1(\overline x)=0$,
\item $l_1(\overline x)\neq 0$ or
\item $o_j(l_1(\overline x),l_2(\overline x))$.
\end{enumerate}
and all the $l_i(\overline x)$ are linear polynomials with coefficients in $R$.
\end{corollary}

\subsection{Reduction of the terms $l(\overline x)=0$}

Let $\varphi$ be an $\overline\LL$-sentence of the form
\[
\exists\overline x\bigwedge \varphi_i(\overline x)
\]
as described in Corollary \ref{corollary:normal-form}. Call any such sentence \emph{normal}. Let $L_\varphi$ be the set of all linear polynomials $l(\overline x)$ such that one of the $\varphi_i(\overline x)$ is the formula $l(\overline x)=0$. We consider the system
\[
\bigwedge_{l(\overline x)\in L_\varphi}l(\overline x)=0
\]
of linear equations as a system over $F$. Using basic linear algebra tools, there exists an algorithm that checks whether this system has a solution over $F$, and if this answer is positive returns a matrix $A$ and a vector $\overline c$, both with entries in $F$, such that
\[
\left<F;0,1,+,\cdot\right>\models \forall \overline y \left(\bigwedge_{l(\overline x)\in L_\varphi}l(A\overline y+\overline c)=0\right)
\]
and
\[
\left<F;0,1,+,\cdot\right>\models \forall \overline x \exists \overline y\left(\bigwedge_{l(\overline x)\in L_\varphi}l(\overline x)=0\rightarrow \overline x=A\overline y+\overline c\right).
\]

Call this algorithm $\texttt{Lin}$, and this application $\texttt{Lin}(\varphi)$.

\begin{lemma} \label{lemma:linear}
Let $\varphi$ be a normal $\overline\LL$-sentence. If $\texttt{Lin}(\varphi)$ has a negative answer, then $\overline\MM_S\not\models \varphi$.
\end{lemma}

If $L_\varphi$ is empty, we consider $\texttt{Lin}(\varphi)$ as having a positive answer, with $A$ the identity matrix and $\overline c$ the zero vector.

\subsection{Translating the problem into completions} \label{sect:translate}

Fix a prime ideal $\pp\in T$. Given a normal $\overline\LL$-sentence $\varphi$ such that $\texttt{Lin}(\varphi)$ returns a positive answer and a pair $(A,\overline c)$, we aim to effectively construct an $\LL_{\ord}$-sentence $\varphi_\pp$ such that
\[
\O_{K,S}\models \varphi\textrm{ implies } \bf{K}_\pp\models \varphi_\pp.
\]

We will show that the following formula $\varphi_\pp$ works: 
\[
\exists\overline y\left( \bigwedge \varphi_{\pp,i}(\overline y) \wedge \ord( A\overline y+\overline c)\ge 0\right)
\]
where $\ord$ of a tuple being nonnegative means that every entry has nonnegative order, and: 
\begin{enumerate}[(i)]

\item For each $i$ such that $\varphi_i(\overline x)$ is of the form $l_1(\overline x)\| l_2(\overline x)$, let $\varphi_{\pp,i}(\overline y)$ be
\[
\exists z_{i_1} \exists z_{i_2} \left(l_1(A\overline y+ \overline c)\cdot z_{i_1}=l_2(A\overline y+\overline c)\wedge \ord(z_{i_1})\ge 0\wedge l_2(A\overline y+ \overline c)\cdot z_{i_2}=1\right).
\]

\emph{Rationale\,: we aim to study divisibility locally at $\pp$ -- thus $l_1\mid l_2$ turns into $\ord(l_1)\le\ord(l_2)$. The introduction of the symbol $\|$ looks to exclude the trivial divisibility $0\mid 0$; as we are looking for a sentence to be interpreted over a field, this can be expressed as $l_2$ being a unit.}

\item For each $i$ such that $\varphi_i(\overline x)$ is of the form $l_1(\overline x)=0$, let $\varphi_{\pp,i}(\overline y)$ be $0=0$.

\emph{Rationale\,: the satisfaction of linear equations has already been considered in the previous section. As we are assuming that $\texttt{Lin}(\varphi)$ has a positive answer, linear equations can be replaced with tautologies.}

\item For each $i$ such that $\varphi_i(\overline x)$ is of the form $l_1(\overline x)\neq 0$, let $\varphi_{\pp,i}(\overline y)$ be
\[
\exists z_i\left(l_1(A\overline y+\overline c)\cdot z_i=1\right).
\]
%

\item For each $i$ such that $\varphi_i(\overline x)$ is of the form $o_j(l_1(\overline x),l_2(\overline x))$, we split into cases. If $\pp=\pp_j$, let $\varphi_{\pp,i}(\overline y)$ be $$\ord(l_1(A\overline y+ \overline c))>\ord(l_2(A\overline y+ \overline c)).$$ Otherwise, let $\varphi_{\pp,i}(\overline y)$ be $0=0$.

\emph{Rationale\,: each predicate $o_i$ refers to the order at one of the prime ideals in $T$. For the one that corresponds to $\pp$, we keep the same meaning by replacing it with $\ord$; for the other predicates, they will not be interpreted locally at $\pp$ and as such are replaced with a tautology.}

\end{enumerate}

\begin{proposition} \label{prop:local}
Let $\varphi$ be a normal $\overline\LL$-sentence such that $\texttt{Lin}(\varphi)$ has a positive answer. If $\O_{K,S}\models\varphi$, then $\bf{K}_\pp\models\varphi_\pp$.
\end{proposition}
\begin{proof}
This comes directly from the construction of $\varphi_\pp$, as $\O_{K,S}\subseteq \bf{K}_\pp$ and for every $x\in \O_{K,S}$ we have $\ord_\pp(x)\ge 0$.
\end{proof}

The following theorem gives the other direction. 

\begin{theorem} \label{thm:main}
Let $\varphi$ be a normal $\overline\LL$-sentence such that $\texttt{Lin}(\varphi)$ has a positive answer. If $K_\pp\models\varphi_\pp$ for every $\pp\in T$, then $\O_{K,S}\models\varphi$.
\end{theorem}

\begin{proof}
Write $\varphi$ as
\[
\exists\overline x\bigwedge \varphi_i(\overline x)
\]
in normal form, and write each $\varphi_\pp$ as
\[
\exists\overline y\left( \bigwedge \varphi_{\pp,i}(\overline y) \wedge \ord(A\overline y+\overline c)\ge 0\right)
\]
as detailed before.

For each $\pp\in T$, let $\overline y_\pp$ be a tuple of elements of $K_\pp$ such that $K_\pp\models \varphi_{\pp,i}(\overline y_\pp)$ for every index $i$. We also choose integers $n_\pp$ in the following way\,:
\begin{enumerate}[(i)]
\item If $\varphi_{\pp,i}$ is of the form
\[
\exists z_{i_1} \exists z_{i_2} \left(l_1(A\overline y+\overline c)\cdot z_{i_1}=l_2(A\overline y+\overline c)\wedge \ord(z_{i_1})\ge 0\wedge l_2(A\overline y+\overline c)\cdot z_{i_2}=1\right),
\]
let $M$ be the maximum of $\{\ord_\pp \left(l_1(A\overline y_\pp+\overline c)\right),\ord_\pp\left(l_2(A\overline y_\pp+\overline c)\right)\}$. By continuity, there exists an integer $n$ such that if $\ord_\pp(\overline y-\overline y_\pp)>n$ then 
$$
\ord_\pp\left((l_1(A\overline y+\overline c)-l_1(A\overline y_\pp+\overline c) \right)>M
$$ 
and 
$$
\ord_\pp\left((l_2(A\overline y+\overline c)-(l_2(A\overline y_\pp+\overline c) \right)>M.
$$ 
Set $n_{\pp,i}$ to be such $n$.

\item If $\varphi_{\pp,i}$ is $0=0$, pick $n_{\pp,i}$ equal to $1$.

\item If $\varphi_{\pp,i}$ is of the form
\[
\exists z_i\left(l_1(A\overline y+\overline c)\cdot z_i=1\right),
\]
we pick, by continuity, $n_{\pp,i}$ such that $\ord_\pp(\overline y-\overline y_\pp)> n_{\pp,i}$ implies 
$$
\ord_\pp\left(l_1(A\overline y+\overline c)-l_1(A\overline y_\pp+\overline c) \right)>\ord_\pp \left(l_1(A\overline y_\pp+\overline c)\right).
$$

\item If $\varphi_{\pp,i}$ is of the form
\[
\ord(l_1(A\overline y+\overline c))>\ord(l_2(A\overline y+\overline c)),
\]
we choose by continuity an integer $n_{\pp,i}$ such that if $\ord_\pp(\overline y-\overline y_\pp)>n_{\pp,i}$ then both $\ord_\pp\left((l_1(A\overline y+\overline c)-l_1(A\overline y_\pp+\overline c) \right)$ and $\ord_\pp\left((l_2(A\overline y+\overline c)-(l_2(A\overline y_\pp+\overline c) \right)$ are greater than the maximum of $\ord_\pp \left(l_1(A\overline y_\pp+\overline c)\right)$ and $\ord_\pp\left(l_2(A\overline y_\pp+\overline c)\right)$.

\item Pick, by continuity, an integer $N_\pp$ such that $\ord_\pp(\overline y-\overline y_\pp)>N_\pp$ implies $\ord_\pp(A\overline y-A\overline y_\pp)\ge 0$.

\item Let $n_\pp$ be the maximum amongst $N_\pp$ and the $n_{\pp,i}$.
\end{enumerate}

By the Artin-Whaples Approximation Theorem, there exists a tuple $\overline y_0$ of elements of $K$ such that, for each $\pp\in T$, $\ord_\pp(\overline y_0-\overline y_\pp)>n_\pp$. Write $\overline x_0=A\overline y_0+\overline c$. We claim that a) every entry of $\overline x_0$ is in $\O_{K,S}$ and b) for every $i$, $\O_{K,S}\models\varphi_i(\overline x_0)$.

For the first claim\,: for each $\pp\in S$, by definition of $n_\pp$ we have 
$$
\ord_\pp\left(\overline x_0-(A\overline y_\pp+\overline c)\right)\ge 0
$$ and, since $K_\pp\models \varphi_\pp(\overline y_\pp)$,
$$
\ord_\pp(A\overline y_\pp+\overline c)\ge 0.
$$ 
Hence $\ord_\pp(\overline x_0)\ge 0$. As this holds for every prime in $S$, the entries of $\overline x_0$ are all in $\O_{K,S}$.

For the second claim\,:
\begin{enumerate}[(i)]
\item Recall that if $\varphi_i$ is of the form $l_1(\overline x)\| l_2(\overline x)$, then each $\varphi_{\pp,i}$ is of the form
\[
\exists z_{i_1} \exists z_{i_2} \left(l_1(A\overline y+ \overline c)\cdot z_{i_1}=l_2(A\overline y+\overline c)\wedge \ord(z_{i_1})\ge 0\wedge l_2(A\overline y+ \overline c)\cdot z_{i_2}=1\right).
\]
By the definition of each $n_\pp$, we have $\ord_\pp(l_1(\overline x_0))=\ord_\pp(l_1(A\overline y_\pp+\overline c))$ and $\ord_\pp(l_2(\overline x_0))=\ord_\pp(l_2(A\overline y_\pp+\overline c))$. Hence $\ord_\pp(l_1(\overline x_0))\le \ord_\pp(l_2(\overline x_0))$ and $l_2(\overline x_0)\neq 0$, and thus $\O_{K,S}\models l_1(\overline x_0)\| l_2(\overline x_0)$.

\item If $\varphi_i$ is of the form $l_1(\overline x)=0$, by construction $l_1(A\overline y+\overline c)=0$ for any $\overline y$ in $K$, in particular it holds for $\overline y_0$.

\item If $\varphi_i$ is of the form $l_1(\overline x)\neq 0$, then $\varphi_{\pp,i}$ is
\[
\exists z_i\left(l_1(A\overline y+\overline c)\cdot z_i=1\right),
\]
and as $\ord_\pp(\overline y_0-\overline y_\pp)>n_\pp$ we have $\ord_\pp(l_1(\overline x_0))=\ord(l_1(A\overline y_\pp+\overline c))<\infty$, and thus $\O_{K,S}\models l_1(\overline x_0)\neq 0$.

\item If $\varphi_i$ is of the form $o_j(l_1(\overline x),l_2(\overline x))$, then $\varphi_{\pp_j,i}$ is
\[
\ord(l_1(A\overline y+\overline c))>\ord(l_2(A\overline y+\overline c)).
\]
As in the first item of this list, as $\ord_{\pp_j}(\overline y_0-\overline y_{\pp_j})>n_{\pp_j,i}$ we have
\[
\ord_{\pp_j}(l_1(\overline x_0))=\ord_{\pp_j}(l_1(A\overline y_{\pp_j}+\overline c))>\ord_{\pp_j}(l_2(A\overline y_{\pp_j}+\overline c))=\ord_{\pp_j}(l_2(\overline x_0)).
\]
\end{enumerate}
\end{proof}


\section{Undecidability results}\label{undres}

In  this section we deal with the dual case where $S$ is a finite set of places containing the primes at infinity, and $K$ is a finite separable extension of $F=\mathbb{F}_p(t)$, where $p$ is any rational prime distinct from $2$. 

We start with fixing some notation. 
\begin{itemize}
\item $\mathcal {M}_{S}= (\O_{K,S}, 0,1, +, |, \cdot t )$.
\item $O^*_{K,S}=O_{K,S}\setminus\{0\}$.
\item $O^\times_{K,S}$ is the group of units of $\O_{K,S}$.
\item $S_{\rm fin}$ is the set of finite primes of $S$ and $s_{\rm f}$ is the cardinal of $S_{\rm fin}$. 
\end{itemize}

\subsection{Definability of ``to be distinct''}

We start by proving that the relation ``different from $0$'' is positive-existentially definable in $ \O_{K,S}$.  In order to do this, we need to recall some results and a few definitions.

\begin{definition}
Let $M_K^\star\subseteq M_K$ denote the set of all non-archimedean places of $K$, and let  $\Sigma\subseteq M_K^\star$. The \emph{Dirichlet density} of $\Sigma$ is
$$
\delta_{\rm Dir}(\Sigma)=\lim_{s\to 1^+}\frac{\sum_{\pp\in\Sigma} q_\pp^{-s}}{ \sum_{\pp\in M_K^\star} q_\pp^{-s}},
$$
if this limit exists, where $q_\pp$ is the size of the residue field at $\pp$.
\end{definition}

Note that if $\Sigma_1$ and $\Sigma_2$ are disjoint and do have Dirichlet density, then the density of the union is just the sum of the densities. We could not find a proper reference for the following easy lemma (see for instance \cite[Ex. 2.3]{Con04}). 

\begin{lemma}\label{lemmaconrad}
Let $K/F$ be a finite separable extension of global fields. The set $\Sigma$ of non-archimedean places $\pp$ on $K$ unramified over $F$, and such that the residue field degree $f(\pp)=[k(\pp)\colon k(P)]$ is $1$ (where $P$ is the prime below $\pp$), has Dirichlet density $1$.
\end{lemma}

If $W$ is a finitely generated subgroup of $K^\times$, of (torsion-free) rank $r\geq 1$, and $\pp$ is a non-archimedean prime, we let $\psi_\pp\colon W\rightarrow k(\pp)^\times$ denote the quotient map to the unit subgroup of the residue field, and $\Sigma(K,W)$ the set of primes $\mathfrak p$ such that:
 \begin{enumerate}
 \item $\ord_{\mathfrak p}(w)=0$ for all $w\in W$, and
 \item the index of $\psi_\pp(W)$ in  $k(\pp)^\times$ divides $2$.
 \end{enumerate}     

In \cite[Theorem 3.1]{Len77}, Lenstra proves in particular: 

\begin{theorem}[Lenstra]\label{thmlenstra}
If $K$ is a function field in one variable over a finite field then the set $\Sigma(K,W)$ has Dirichlet density strictly less than $1$. 
\end{theorem}
 
\begin{lemma}\label{lemlenstra} There exist an irreducible polynomial  $q\in \mathbb F_p[t]$ and a non-zero polynomial $b\in \mathbb F_p[t]$ of degree less than the degree of $q$, such that $qx+b$  is never a unit in $ \O_{K,S}$ as $x$ varies over $ \O_{K,S}$.
\end{lemma}
\proof Let $W$ be the group of units $\O_{K,S}^\times$, and let $\Sigma$ be the set of primes $\pp$ in $M_K$ such that $f(\pp)=1$. Since the density of $\Sigma$ is $1$ and that of the complement of $\Sigma(K,W)\cup S_{\rm fin}$ is positive, their intersection is not empty. Let $\pp\in\Sigma\setminus (\Sigma(K,W)\cup S_{\rm fin})$. So in particular we have $f(\pp)=1$. Choose $q$ to be the monic irreducible polynomial generating the prime below $\pp$. Since $f(\pp)=1$, we have $k(\pp)^\times=k(q)^\times$. Since $\mathfrak p\notin \Sigma(K,W)\cup S_{\rm fin}$, the index of $\psi(W)$ in $ k({\mathfrak p})^\times$ does not divide $2$. In particular $\psi(W)\subsetneq k(q)^\times$. Let $b$ be any polynomial of degree less than the degree of $q$ whose residue class modulo $q$ lies in $k(q)^\times\setminus \psi(W)$. It follows from the construction  that  $q$ and $b$ are as required. 
\endproof
 

The following theorem can be found in \cite[Thm. 4.7]{R02}.

\begin{theorem}[Kornblum] Let $a,b\in \mathbb F_p[t]$ be two relatively prime polynomials. If $a$ has positive degree, then the set 
$$
\Gamma=\{\,q\in \mathbb F_p[t]\colon  q\equiv b\pmod a,\ q \ \rm{is\ irreducible}\,\}
$$
has positive Dirichlet density. In particular, $\Gamma$ is infinite.   
\end{theorem}

\begin{lemma}\label{dif}
The relation $\neq$ is positive-existentially definable in the structure $\mathcal M_S$.
\end{lemma}
\proof
Let $q$ and $b$ be given by Lemma \ref{lemlenstra}. The formula
$$
\psi_{\neq}(y):\exists A,B,x(y\mid A\wedge qx+b\mid  B\wedge A+B=1)
$$
positive-existentially defines the relation ``$y\neq 0$'' in $\mathcal{M}_S$.

First note that the formula $\psi_{\neq}(y)$ translates to  ``There exist $r,s,x\in \O_{K,S}$ such that $ry+s(qx+b)=1$'' in $\mathcal M_S$.   

If  $y=0$, then the formula is false, since by Lemma \ref{lemlenstra}, $qx+b$ is never a unit in $\O_{K,S}$. 

For the other implication suppose $y\neq 0$. We may assume that $y\in \O_K$ (otherwise, multiply the $r$ that we will find by the denominator of $y$).  Let $\pp_1^{k_1}\cdots \pp_m^{k_m}$ be the prime factorization  of the ideal $y\O_{K}$, and let $P_1,\dots, P_m$ be the primes below $\pp_1,\dots, \pp_m,$ respectively.   Since $q$ and $b$ are relatively prime, by Kornblum's  theorem there exists $x$ such that $qx+b$ is irreducible, and coprime with  $P_1,\dots, P_m,$. It follows that the ideals $y\O_K$ and $(qx+b)\O_K$ are comaximal. Hence, there are $r,s \in \O_K$ so that $ry+s(qx+b)=1$ as required.     \endproof

\subsection{Defining the square function restricted to units}

In this section we show that the square function restricted to units is positive-existianlly definable in $\mathcal{M}_S$. In order to do this, we need a few more definitions and notation. 
\begin{itemize}
\item $P_1,\dots, P_k$ are the irreducible elements in $\mathbb F_p[t]$ below the primes of $S_{\rm fin}$.
\item $S_\infty=\{\infty_1,\dots,\infty_{s_\infty}\}$ is the set of infinite primes of $S$ and $s_{\infty}$ is the cardinal of $S_\infty$.
\item $s=s_\infty+s_{\rm f}$. 
\item $c_K$ is the cardinal of the field of constants $C_K$.
\item $r$ is an upper bound for the amount of primes above each $P_i$.
\item $I=\{0,\dots, 2r+1\}^k$.
\item If $\alpha=(\alpha_1,\dots,\alpha_k)\in \mathbb N^k$ is a multi-index, we write $P^\alpha=\prod_{i=1}^kP_i^{\alpha_i}$. 
\item For $i=1,\dots, c_K$ and $j=1,\dots,2s+1$, $q_{i,j}\in\mathbb F_p[t]$ are irreducible elements such that: 
\begin{itemize}
\item $q_{i,j}$ does not lie below any of the primes in $S_{\rm fin}$;
\item for each $j$, all the elements in $\{q_{i,j}\colon i=1,\dots, c_K\}$ have the same degree; 
\item for each $i$, the elements in $\{q_{i,j}\colon j=1,\dots, 2s_\infty+1\}$ have pairwise distinct degrees. 
\end{itemize}
\end{itemize}

\begin{lemma}\label{fakepigeonhole}
For every $x,y \in \O^\times_{K,S}$, there exist a multi-index $\beta\in I$ and an index $j_0$ such that $\ord_\pp(P^\beta x)\ne\ord_\pp q_{i,j_0}$ and $\ord_\pp(P^\beta y)\ne\ord_\pp q_{i,j_0}x$ for every prime $\pp$ in $S$ (including the primes at infinity).
\end{lemma}
\proof
Given $x,y \in \O^\times_{K,S}$, for each  $1\leq\ell \leq k$, we choose 
$$
\beta_\ell\in\{1,2,3,\dots,2r+1\}\setminus \left(\left\{\frac{-\ord_\pp(x)}{e_\pp}\colon \pp \dib P_i\right\}\cup\left\{\frac{\ord_\pp(x)-\ord_\pp(y)}{e_\pp}\colon \pp  \dib   P_i\right\}\right),
$$
arbitrarily. Given $\pp\in S_{\rm fin}$, let $P_i$ be the prime below $\pp$. We have
$$
\ord_\pp(P^\beta x)=\ord_\pp(x)+\sum_\ell \beta_\ell\ord_\pp(P_\ell)
=\ord_\pp(x)+\beta_i e_\pp\ne 0, 
$$ 
hence $\ord_\pp(P^\beta x)\ne0$. Similarly, we have $\ord_\pp(P^{\beta}y)\ne \ord_\pp(x) $ for all $\pp\in S_{\rm fin}$. Note that in the worst case, there is just one choice for $\beta$. Since $q_{i,j}$ do not lie below the primes in $S_{\rm fin}$, we have $\ord_\pp q_{i,j}=0$ for all those primes, so in particular $\ord_\pp(P^\beta x)\ne\ord_\pp q_{i,j}$ and $\ord_\pp(P^\beta y)\ne\ord_\pp q_{i,j}x$. 

Let us write $d_j=\ord_\infty(q_{i,j})$ (namely, the opposite of the degree of $q_{i,j}$). Choose $j_0$ to be any index so that $e_{\pp}d_{j_0}$ is distinct from both $\ord_\pp(P^{\beta}x)$ and $\ord_\pp(P^{\beta}y)-\ord_\pp(x)$ for every prime $\pp$ at infinity. This is possible because there are $2s_\infty+1$ possible pairwise distinct degrees among the $d_j$, and there are $s$ primes at infinity. With this choice of $j_0$, for every $i$, we also have $\ord_\pp(P^\beta x)\ne\ord_\pp q_{i,j_0}$ and $\ord_\pp(P^\beta y)\ne\ord_\pp q_{i,j_0}x$ for the primes at infinity.
\endproof

\begin{proposition}\label{squareunits}
The set $\SQ_u=\{(x,y)\colon x,y \in \O^\times_{K,S},\ y=x^2\}$ is positive-existentially definable in $\mathcal{M}_S$.
\end{proposition}
\proof
We shall prove that the following  formula works.
$$
\Sq_u(x,y)\colon x\mid 1\wedge y\mid 1\wedge\bigwedge_{\substack{\alpha\in I\\i=1,\dots,c_K\\ j=1,\dots,2s_\infty+1}}\ass(P^{\alpha} x+q_{i,j},P^{\alpha} y+q_{i,j}x).
$$
If $y=x^2$ for some $x,y\in\O^\times_{K,S}$, then $\Sq_u(x,y)$ holds trivially in $\mathcal M_S$. 

Assume  that $\Sq_u(x,y)$ holds in $\mathcal M_S$, and let $\beta$ and $j_0$ be given by Lemma \ref{fakepigeonhole}. Let $u_{i}$ be units such that
\begin{equation}\label{equprin3-1}
P^{\beta}x+q_{i,j_0}=u_{i}(P^{\beta}y+q_{i,j_0}x).
\end{equation}
Note that by our definition of the $q_{i,j}$, we have $P^{\beta}x+q_{i,j_0}$, $P^{\beta}x$, $P^{\beta}y$ and $q_{i,j_0}$ are different from $0$ for every $i$. For every prime $\pp$ in $S$, since $\ord_\pp (P^{\beta}x)\neq \ord_\pp (q_{i,j_0})$ and $\ord_\pp(P^{\beta}y)\neq\ord_\pp(q_{i,j_0}x)$ by Lemma \ref{fakepigeonhole}, we have 
$$
\begin{aligned}
\ord_\pp(u_{i})&=\ord_\pp(u_{i}(P^{\beta}y+q_{i,j_0}x))-\ord_\pp(P^{\beta}y+q_{i,j_0}x)\\
&=\ord_\pp(P^{\beta}x+q_{i,j_0})-\min \{\ord_\pp(P^{\beta}y),\ord_\pp(q_{i,j_0}x)\}\\
&=\min \{\ord_\pp(P^{\beta}x),\ord_\pp(q_{i,j_0})\}-\min \{\ord_\pp(P^{\beta}y),\ord_\pp(q_{i,j_0}x)\},
\end{aligned}
$$
and the latter expression does not depend on $i$, because by definition of $q_{i,j}$ we have $\ord_\pp q_{i,j}=\ord_\pp q_{\ell,j}$ for every $i$ and $\ell$. In particular, we have $\ord_\pp (u_{i})=\ord_\pp (u_{\ell})$ for every prime $\pp$ in $S$ and every $i$ and $\ell$. This implies that $u_i$ and $u_\ell$ have the same zeros and poles, hence they differ by a constant unit. Since there are $c_K-1$ constant units, it follows from the pigeonhole principle that there exists $i\ne \ell$ so that $u_i=u_{\ell}=u$. 

From Equation \eqref{equprin3-1} for $i$ and $\ell$, after subtracting the sides of the equations, we obtain $1=ux$.  Substituting in the first equation we get $y=x^2$, as required. 
\endproof

Next, we need to be sure that there are  large enough units. 


\begin{lemma}\label{lemunitinforder}
For every $x\in O^*_{K,S}$ there exists a unit $\varepsilon$ of infinite order such that $x \dib \varepsilon-1$.  
\end{lemma}
\proof
Let $x\in O^*_{K,S}$ be given. Without loss of generality we may assume that $x$ is not a unit and that $x\in \O_{K}$. Let $(x)=\pp_1^{\alpha_1}\cdots \pp_\ell^{\alpha_\ell}$ be the prime factorization  of the ideal $x\O_K$. 

If the set $\{\pp \in S\colon \pp \dib \infty\}$ has cardinality bigger than $2$, then we can choose a unit $\varepsilon_0\in \mathcal O^\times_K$ of infinite order. In that case, since $\O_K/(x)$ is a finite ring, there exist $m<n$ such that $\varepsilon_0^m\equiv \varepsilon_0^n\pmod{(x)}$, and since $\varepsilon_0$ is a unit, we have $\varepsilon_0^{n-m}\equiv 1\pmod{(x)}$. Thus, $\varepsilon=\varepsilon_0^{n-m}$ is as required.

Otherwise, choose $\pp\in S$ which is not a prime at infinity (such a $\pp$ exists as we assumed that $S$ has at least two elements). Let $h_K$ denote the ideal class number. As each ideal $\pp_i^{\alpha_ih_K }$ is principal, we can write $x^{h_K}=x_1\cdots x_\ell$, where $(x_i)=\pp_i^{\alpha_ih_K}$. We may assume, without loss of generality, that $\pp_i$ does not belong to $S$ for any $i$. Let $\varepsilon_0$ be any generator of $\pp^{h_K}$.  Since $\mathcal O_K/(x^{h_K})$ has finite order, and the class of $\varepsilon_0$ is not a zero-divisor in the quotient, we can conclude as in the previous case. 
\endproof

\begin{rmk}
	Note that since the torsion part of  $\O^\times_{K,S}$ is finite, there is a positive existential $\mathcal{L}_{\Div,t}$-formula $\varphi_\infty(x)$ so that $\varphi_\infty(u)$ holds in $\mathcal{M}_S$ if and only if $u$ is a unit of infinite order.
\end{rmk}

\subsection{Defining the square function}

The rest of the section will be devoted to prove that the square function is positive existentially definable over $\mathcal{M}_S$. In order to do this, we will introduce some notion of relative largeness between non-zero elements and units of $\O^*_{K,S}$. 

\begin{lemma}\label{keylemma}
There is a constant $C_1\geq 0$ so that given  any $x\in \mathcal O_{K,S}^*$, there exists $a\in  O_{K,S}^*$  and $b\in \O^*_{K}$ such that $x=\frac{a}{b}$, $\ord_\pp(a)\geq -h_K$ for every $\pp\in S_{\rm fin}$, and $\ord_{\qq}(b)\leq C_1$ for every $\qq\in S_\infty$.  Moreover, for every $\pp\in S_{\rm fin}$, if $\ord_\pp(a)>0,$ then $\ord_\pp(b)=0$.
\end{lemma}
\proof
Let $D(x)=\{\pp\in S_{\rm fin}\colon \ord_\pp(x)<0\}$  and let $N(x)=\{\pp\in S_{\rm fin}\colon \ord_\pp(x)>0\}$.  For each   prime $\pp\in D(x)$, set $\beta_\pp\leq |\ord_\pp(x))|$ be the greatest integer such that $\pp^{\beta_\pp}$ is a principal ideal. If there is no such $\beta_\pp$, set $\beta_\pp=0$. Let $b'\in\O_K$ such that $b'\O_K=\prod_{\pp\in D(x) }\pp^{\beta_\pp}$, and let $a'=b'x$. We have then 
$$
\ord_\pp(a')=\beta_\pp+\ord_\pp(x)\geq -h_K, 
$$
as otherwise we would have $\beta_\pp+h_K<-\ord_\pp(x)=|\ord\pp(x)|$, which would contradict the choice of $\beta_\pp$. 

The $b$ we are looking for will be a carefully chosen associate of $b'$.  If $s_\infty=1$, we choose $b=b'$. Let us assume $s_\infty\geq 2$. Let $i\colon \O_K\to \mathbb Z^{s_\infty}$ be the map given by $i(x)=(\ord_{\infty_i}(x))_i$.  By the analogue of Dirichlet unit theorem for function fields (see \cite[Prop. 14.1 (b), p. 243]{R02}), we have that the image of the unit group $\Lambda=i(\O_K^\times)$ is a full sublattice of the lattice $\{ (c_1,\dots,c_{s_\infty})\colon \sum_{i=1}^{s_\infty} c_i=0\}$. Notice that if $N$ is a large enough positive integer (one can take the product of $\sqrt{s_\infty}$ by the diameter of a disc containing the fundamental domain), then for any $x\in \O_K$ such that $\sum_{i=1}^{s_\infty} \ord_{\infty_i}(x)\leq -N$, we can find an element $(c_1,\dots,c_{s_\infty})\in \Lambda$ such that $c_i+\ord_{\infty_i}(x)\leq 0$ for any $i=1,\dots,s_\infty$. If $\sum_i \ord_{\infty_i}(b')\leq -N$, there exists $c_1,\dots,c_{s_\infty}$ such that $c_i+\ord_{\infty_i}(b')\leq 0$, so there exists a unit $u$ such that $i(u)=(c_1,\dots,c_{s_\infty})$. In that case, we choose $b=ub'$ and we can take $C'=0$. Assume $\ell=\sum_i \ord_{\infty_i}(b')>-N$. There exists a unit $i(u)=(c_1,\dots,c_{s_\infty})$ such that $i(ub')$ lies within a fundamental domain of $\Lambda+\ell$. Since there are finitely many possible values for $\ell$, and for each given $\ell$, there are finitely many points within a fundamental domain of $\Lambda+\ell$, we can choose the maximum of the coordinates of these points as $\ell$ varies. In that case, take $C''$ to be this maximum. The maximum between $C'$ and $C''$ is the constant $C_1$ that we were looking for. 
\endproof  


We now need a few more definitions. For $x\in K^*$, we define the divisors of zeros and poles as follows:

$$
\Div_0(x)=\sum_{\substack{\ord_\pp(x)>0}} \ord_{\pp}(x)\pp \qquad{\rm and}\qquad  \Div_\infty(x)=-\sum_{\substack{\ord_\pp(x)<0}} \ord_{\pp}(x)\pp
$$
and define $\deg_K(\mathfrak{a})=\sum_{\pp\in M_K} f(\pp)\ord_{\pp}(\mathfrak{a})$. 
Let 
$$
\ord_{\rm min} (x)=\min\{\ord_{\infty_i}(x)\colon i=1,\dots,s_\infty\}.
$$

We now recall a few definitions and notation from \cite[Ch. 7, pp. 82--83]{R02}. 

Let $\mathcal{D}_K, \mathcal{D}_F$ denote the divisor group of $K$ and $F$, respectively. 
\begin{itemize}
	\item Let $i_{K/F}\colon \mathcal{D}_F\to \mathcal{D}_K$ be the morphism defined  by $i_{K/F}(P)=\sum_{\substack{\pp\dib P}}e(\pp)\pp$ for all places $P\in M_F,$ and extended linearly.
	\item $\deg_K\colon \mathcal{D}_K\to \mathbb{Z}$ be the morphism defined  by $\dim_{C_K}(\O_\pp/\pp)$ for all places $\pp\in M_K$, and extended linearly.
	
\end{itemize}

\begin{lemma} For any $x\in K^*$,
	$$\deg_K(\Div_0(x))=\deg_K(\Div_\infty(x)).$$
\end{lemma}

The following lemma correspond to \cite[Prop. 7.7, pp. 82]{R02}.
\begin{lemma}\label{rosen1}
	Let $\mathfrak{A}\in \mathcal{D}_K$ and $A\in \mathcal{D}_F$. We have
	$$
 \deg_K(i_{K/F}(A))=\frac{[K\colon F]}{[C_K\colon \mathbb{F}_p]}\deg_F(A).
	$$
\end{lemma}

Let $q_j$, with $j=1,\dots, 2s_\infty+1$ be a finite sequence of irreducible polynomials in $\F_p[t]$ not below any prime of $S_{\rm fin}$ and such  that $\ord_\infty( q_j-q_\ell)< -2C_1$, and $ \ord_{\infty}(q_j)\ne \ord_{\infty}(q_\ell)$ for any $j\ne\ell$. Let $q$ be an irreducible polynomial not below any of the primes in $S_{\rm fin}$ such that 
$$
\deg(q)>h_Ks_{\rm f}\frac{[C_K: \F_p]}{[K:F]}D,
$$ 
where 
$$
D=\max\{\deg_K(\pp)\colon \pp \in S\}.
$$

The following Lemma makes precise the idea that the unit $\varepsilon$ in Lemma \ref{lemunitinforder} is much larger than $x$. 


\begin{lemma}\label{lemC2}
Let $x\in \O^*_{K,S} $ be such that $\ord_\pp(x)\ne 0$ for any $\pp\in S_{\rm fin}$, and let $\varepsilon\in \O_{K,S}^\times$ be a unit of infinite order such that $x+q_j | \varepsilon-1$ for $j=1,\dots, 2s_\infty+1$, and $q\dib \varepsilon-1$. If we write $x=\frac{a}{b}$ and $\varepsilon=\frac{u}{v}$, where $a,b$ and $u,v$ are as in Lemma \ref{keylemma}, then there exists a constant $C_2$ such that  
$$
-\min\{ \ord_{\rm min}(a),\ord_{\rm min}(b)\}\leq -s_{\infty} D \ord_{\rm min} (u-v) +C_2.
$$
\end{lemma}
\proof Write $x=\frac{a}{b}$, and $\varepsilon=\frac{u}{v}$, where $a,b$ and $u,v$ are as in the  Lemma \ref{keylemma}.\\
We will split the proof in few claims. First of all, we shall show that $q\dib \varepsilon-1$, implies that $\ord_{\min}(u-v)<0$.\\

\noindent{\bf Claim A:} There is an $i$ such that $\ord_{\infty_i}(u-v)<0$. Hence, $\ord_{\min}(u-v)<0.$\\
\textit{Proof of Claim A:}
Since $q\dib u-v$ we have $\ord_{\mathfrak{q_i}}(q)\leq \ord_{\mathfrak{q_i}}(u-v)$ for all the primes $\mathfrak{q}_1,\dots,\mathfrak{q}_k$ above $q$. This implies
\begin{equation}\label{eqlemC2a}
\begin{aligned}
\deg_K(i_{K/F}(\Div_0(q)))=\\
\sum_{i=1}^{k}\ord_{\mathfrak{q}_i}(q)\deg_K(\qq_i)&\leq \sum_{\substack{\ord_\pp(u-v)>0}}
\ord_{\pp}(u-v)\deg_K(\pp)\\
&
=\deg_K(\Div_0(u-v))=\deg_K(\Div_\infty(u-v))\\
&=-\sum_{\substack{\ord_\pp(u-v)<0}} \ord_{\pp}(u-v)\deg_K(\pp)\\
&=-\sum_{\substack{\pp\in S_{\rm fin}\\\ord_\pp(u-v)<0}} \ord_{\pp}(u-v)\deg_K(\pp)\\&-\sum_{\substack{\pp\in S_\infty\\\ord_\pp(u-v)<0}} \ord_{\pp}(u-v)\deg_K(\pp).
\end{aligned}
\end{equation}
By Lemma \ref{rosen1} we have 
\begin{equation}\label{eqlemC2b}
\deg_K(\Div_0(i_{K/F}(q)))=\frac{[K:F]}{[C_K:\F_p]}\deg(q)>Dh_Ks_{\rm f}
\end{equation}
where the last inequality comes from our choice of $q$. From \eqref{eqlemC2a} and \eqref{eqlemC2b}, we obtain
$$
\begin{aligned}
-\sum_{\substack{\pp\in S_\infty\\\ord_{\pp}(u-v)<0}} \ord_{\pp}(u-v)\deg_K(\pp)&>Dh_Ks_{\rm f}+\sum_{\substack{\pp\in S_{\rm fin}\\\ord_{\pp}(u-v)<0}} \ord_{\pp}(u-v)\deg_K(\pp)\\
&>Dh_Ks_{\rm f}-Dh_Ks_{\rm f}=0.
\end{aligned}
$$
This finishes the proof of Claim A. $\hfill \square$\\


\noindent\textbf{Claim B:} There exists $j_0$ such that 
$$
\ord_{\infty_i}(a)\ne \ord_{\infty_i}(bq_{j_0}), \qquad \ord_{\infty_i}(a+bq_{j_0})\le 0\qquad
 {\rm for\ all\ }\qquad i=1,\dots,s_\infty.
$$
\textit{Proof of Claim B:} Notice that for each $i$, we have 
$$
\ord_{\infty_i}(a+bq_j-a-bq_{\ell}) 
=\ord_{\infty_i}(b(q_j-q_\ell)) = \ord_{\infty_i}(b)+\ord_{\infty_i}(q_j-q_\ell) <0.
$$ 
It follows that $\ord_{\infty_i}(a+bq_j)\leq 0$ except for at most one $j$. By the pigeonhole principle there exists a $J\subseteq\{1,\dots,2s_\infty+1\}$ of cardinality $s_\infty+1$ so that  $\ord_{\infty_i}(a+bq_{j})\leq 0$ for any  $i=1,\dots,s_\infty$ and any $j\in J$. Since $ \ord_{\infty}(q_j)\ne \ord_{\infty}(q_\ell)$ for any $j\ne \ell$, we have $\ord_{\infty_i}(a)\ne\ord_{\infty_i}(bq_{j})$ except for at most one $j\in J$, hence $\ord_{\infty_i}(a+bq_{j})=\min\{\ord_{\infty_i}(a),\ord_{\infty_i}(bq_{j})\}$ except for at most one $j\in J$. Applying the pigeonhole principle again there exists $j_0\in J$ so that 
$$
\ord_{\infty_i}(a+bq_{j_0})=\min\{\ord_{\infty_i}(a),\ord_{\infty_i}(bq_{j_0})\}
$$ 
for any $i=1,\dots,s_\infty$. 
This finishes the proof of Claim B. $\hfill\square$

Notice that 
$$
x+q_{j_0} | \varepsilon-1 \qquad\textrm{if and only if}\qquad a+bq_{j_0} | u-v.
$$ 
thus, we have $\ord_\pp(a+bq_{j_0})\leq \ord_{\pp}(u-v)$ for all $\pp\notin S$. It follows that 

\begin{equation}\label{lemeq4}
 \sum_{\substack{\pp\notin S}}\ord_{\pp}(a+bq_{j_0})\deg_K(\pp)\leq \sum_{\substack{\pp\notin S}} \ord_{\pp}(u-v)\deg_K(\pp)
\end{equation}
 
We aim to obtain an upper bound for (\ref{lemeq4}).\\

\noindent\textbf{Claim C:} $$\sum_{\substack{\pp\notin S}}\ord_{\pp}(u-v)\deg_K(\pp)\le Dh_K s_{\rm f}-s_\infty D\ord_{\min}(u-v).$$
 
\textit{Proof of Claim C:} In order to get an upper bound, we proceed as follows.

\begin{equation}
\begin{aligned}
\sum_{\substack{\pp\notin S}} \ord_{\pp}(u-v)\deg_K(\pp)&\leq 
\sum_{\substack{\ord_\pp(u-v)>0}} \ord_{\pp}(u-v)\deg_K(\pp)\\
&=-\sum_{\substack{\ord_\pp(u-v)<0}} \ord_{\pp}(u-v)\deg_K(\pp)\\
&=-\sum_{\substack{\pp\in S_{\rm fin}\\ \ord_\pp(u-v)<0}}\ord_{\pp}(u-v)\deg_K(\pp) +\\&-\sum_{\substack{\pp\in S_\infty\\ \ord_\pp(u-v)<0}}\ord_{\pp}(u-v)\deg_K(\pp)\\
&\leq Dh_K s_{\rm f}-s_\infty D\ord_{\min}(u-v).
\end{aligned}
\end{equation}
where the last inequality holds as $\ord_{\min}(u-v)<0$.$\hfill\square$\\

Now, we aim to obtain a lower bound for (\ref{lemeq4}).\\

\noindent\textbf{Claim D:}
$$
-s_\infty D\min\{\ord_{\min}(a),\ord_{\min}(b)\}-Dh_K s_{\rm f}\leq \sum_{\substack{\pp\notin S}}\ord_{\pp}(a+bq_{j_0})\deg_K(\pp).
$$

\textit{Proof of Claim D:}
Recall that by hypothesis $\ord_{\pp}(x)\ne 0$ for all $\pp\in S_{\rm fin}$. It follows, from our choice of $a,b$ and Claim B that $\ord_\pp(a+bq_{j_0})\leq 0 $ for all $\pp\in S_{\rm fin}$ and $\ord_{\infty_i}(a+bq_{j_0})\le 0 $ for $i=1,\dots,s_\infty$. Hence,

$$
\sum_{\substack{\pp\notin S}} \ord_{\pp}(a+bq_{j_0})\deg_K(\pp)=
\sum_{\substack{\ord_\pp(a+bq_{j_0})>0}} \ord_{\pp}(a+bq_{j_0})\deg_K(\pp)$$
$$
=- \sum_{i=1}^{s_\infty}\ord_{\infty_i}(a+bq_{j_0})\deg_K(\infty_i)
-\sum_{\substack{\pp\in S_{\rm fin}}} \ord_{\pp}(a+bq_{j_0})\deg_K(\pp)
$$
$$
\geq-\ord_{\min}(a+bq_{j_0})=-
\min\{\ord_{\min}(a),\ord_{\min}(bq_{j_0})\}
$$
$$
\geq -\min\{\ord_{\min}(a),\ord_{\min}(b)\}-[K:F]\deg(q).\ \ \ \ \ \ \ \square
$$
Finally, from Claim C and D we obtain  
$$
-\min\{\ord_{\min}(a),\ord_{\min}(b)\}+[K:F]\leq Dh_ks_{\rm f}-s_\infty D\ord_{\rm min}(u-v).
$$ 
Setting $C_2=Dh_Ks_{\rm f}-[K:F]$, we obtain 

$$
-\min\{\ord_{\min}(a),\ord_{\min}(b)\} \leq -s_\infty D \ord_{\rm min}(u-v)+C_2.
$$
This finishes the proof of the Lemma. \endproof

We are ready to prove the main theorem of the section. The proof follows the same pattern as before, though much harder. We first introduce some notation (we redefine the letter $I$ that was used before).

\begin{itemize}
\item $I=\{1,\dots, 3r+1\}^k$.
\item $J=\{1,\dots,2r+1\}^k.$
\item Let $r_i$ with $i=1,\dots, s_\infty+1$ be a finite sequence of irreducible polynomials in $\F_p[t],$ not below any prime in $S_{\rm fin}$, with pairwise distinct degrees. 
\item 
Let $s_j$, with $j=1,\dots, s_\infty+1$ be a finite sequence of irreducible polynomials in $\F_p[t]$ not below any prime of $S_{\rm fin}$ and such  that $\ord_\infty( q_j-q_\ell)< -(N_1+1)C_1$, and $ \ord_{\infty}(q_j)\ne \ord_{\infty}(q_{j'})$ for any $j\ne j'$.
\end{itemize}
Choose $N_0$ to be any positive integer such that 
$$
-N_0n>
3 s  D h_K-s s_{\infty} D^2(4n+(-C_3+s D\omega  )
$$
holds for any integer $n\le -1, $ where 
$$
C_3=\max\{\ord_\pp(P^\beta s_{j})\colon i,j=1,\dots, 2s_\infty+1,  \beta \in J , \pp\in S_{\rm fin}\}.
$$
Let $N=\log_2(N_0)+1,$ and $N_1=2^{N-1}$.
\begin{theorem}
	The set 
	$$
	SQ=\{(x,y)\in \O_{K,S}\colon y=x^2\}
	$$
	is positive-existentially definable in the structure  $\mathcal M_S.$
\end{theorem}
\proof We claim that the following formula $\varphi_{\rm sq}(x,y)$ defines the set $SQ$.
\begin{equation}
(x=0\wedge y=0)\vee \bigvee_{\substack{\alpha\in I\\ j=1,\dots,4s_\infty+1\\ P^\alpha\in\O_{K,S}^\times}}(x=- P^{-\alpha}q_j\wedge y=P^{-2\alpha}q_j^2)
\end{equation}

\begin{equation}
\vee  \exists \varepsilon_1(\varepsilon_1 \dib 1 \wedge \varphi_\infty(\varepsilon_1)) \wedge \psi(x,y)
\end{equation}

where $\psi(x,y)$ is the conjunction of the following formulas:
$$
\psi_1(x,\varepsilon_1): \bigwedge_{\substack{\alpha\in I\\j=1,\dots,2s_\infty+1}} P^{\alpha}x+q_j\dib \varepsilon_1-1
$$

$$
\psi_2(y,\varepsilon_1): \bigwedge_{\substack{\alpha\in I\\j=1,\dots,2s_\infty+1}} P^{2\alpha}y+q^2_j\dib \varepsilon_1-1
$$

$$
\psi_3(\varepsilon_1): q \dib \varepsilon_1-1
$$

$$
\psi_4(\varepsilon_1): \exists \varepsilon_2\dots\varepsilon_N\left(\bigwedge_{i=1}^{N-1} \varepsilon^2_i=\varepsilon_{i+1}\right)
$$

$$
\psi_5(x,y,\varepsilon_N): \bigwedge_{\substack{\alpha\in I\\ \beta\in J\\i=1,\dots,s_\infty+1\\j=1,\dots,s_\infty+1}} P^{\alpha}x+(r_i\varepsilon_{N}+P^\beta s_{j})\dib P^{2\alpha}y-(r^2_i\varepsilon_{N}^2+2r_i\varepsilon_{N} P^\beta s_{j} +P^{2\beta}s_{j}^2).
$$
First we show that if $y=x^2$, then $\varphi_{\rm sq}(x,y)$ holds in $\mathcal{M}_S$. If $x=0$ or $x=-P^{-\alpha}q_j$ for $\alpha\in I$ and $j\in\{1,\dots,4s_\infty+1\}$, then clearly $\varphi_{\rm sq}(x,y)$ holds. So we may assume $x\ne 0$ and  $x\ne -P^{-\alpha}q_j$ for any $\alpha\in I$ and $j\in\{1,\dots,4s_\infty+1\}$, but in this case the result follows immediately from Lemma \ref{lemunitinforder}.
\\

Let us now assume that $\varphi_{\rm sq}(x,y)$ holds. Without loss of generality we may also assume that $x\ne0,$ and  $x\ne -P^{-\alpha}q_j$ for any $\alpha\in I$ and any $j=1,\dots,4s_\infty+1$. Fix $\varepsilon_1$  so that $\psi(x,y)$ holds. 

We will break the rest of the proof into several Claims.\\


\noindent\textbf{Claim A:} There exists $\alpha\in I$ such that $\ord_\pp(P^\alpha x)\ne 0,$ $\ord_\pp(P^{2\alpha} y)\ne 0$, $\ord_\pp(P^\alpha x)\ne \ord_\pp (\varepsilon_N)$ for all $\pp\in S_{\rm fin}$. \\
\textit{Proof of Claim A:}
Given $x,y \in \O^*_{K,S}$, for each  $1\leq\ell \leq k$, we choose 
$$
\alpha_\ell\in I\setminus \left(\left\{\frac{-\ord_\pp(x)}{e_\pp}\colon \pp \dib P_i\right\}\cup\left\{\frac{-\ord_\pp(y)}{2e_\pp}\colon \pp  \dib   P_i\right\}\cup \left\{\frac{-\ord_\pp(\varepsilon_N)}{e_\pp}\colon \pp  \dib   P_i\right\} \right),
$$
arbitrarily. $\hfill\square$

From now on, in order to simplify notation, we will write $X$ and $Y$ instead of $P^\alpha x$ and $P^{2\alpha}y$. Since $\psi_1(x,\varepsilon_1)$, $\psi_2(y,\varepsilon_1)$ and $\psi_3(\varepsilon_1)$ hold, from Lemma \ref{lemC2} we have 
\begin{equation}\label{lasteq}
\begin{split}
-\min\{ \ord_{\rm min}(a),\ord_{\rm min}(b),\ord_{\rm min}(c),\ord_{\rm min}(d)\} \\
\leq -s_{\infty}
D\min\{\ord_{\min}(u),\ord_{\min}(v)\},
\end{split}
\end{equation}
where $X=\frac{a}{b}$, $Y=\frac{c}{d}$ and $\varepsilon_1=\frac{u}{v}$ are as in Lemma \ref{keylemma}.


From $\psi_5$, for every choice of $i=1,\dots,2s_\infty+1$, $j=1,\dots,4s_\infty+1$ and $\beta\in J$, we have 
$$
X+ (r_i\varepsilon_N+P^\beta s_{j})\ |\ Y-(r_i^2\varepsilon_N^{2}+2r_i\varepsilon_N P^\beta s_{j} +P^{2\beta}s_{j}^2),
$$ 
and  
$$
X+ (r_i\varepsilon_N+P^\beta s_{j})\ | \ X^2-(r_i^2\varepsilon_N^{2}+2r_i\varepsilon_N P^\beta s_{j} +P^{2\beta}s_{j}^2)
$$ 
hence $X+ (r_i\varepsilon_N+P^\beta s_{j})\  | \ Y-X^2$. 
In order to get a contradiction, suppose $Y\ne X^2$. Write $N_1=2^{N-1}$. Since $\varepsilon_N=\varepsilon_1^{2^{N-1}}=\varepsilon_1^{N_1}$ it follows that 
$$ 
av^{N_1}+r_ibu^{N_1}+bv^{N_1}P^\beta s_{j}\ \dib \ b^2c-a^2d,
$$
hence 
\begin{equation}\label{eqtocontradict}
\ord_\pp(av^{N_1}+r_ibu^{N_1}+bv^{N_1}P^\beta s_{j})\leq \ord_\pp(b^2c-a^2d),
\end{equation}
for any $\pp \notin S$. 

The following two claims are proven like Lemma \ref{fakepigeonhole}. \\

\noindent\textbf{Claim B1:} There exists $\beta\in J$ such that $\ord_\pp(P^\beta)\ne \ord_\pp(X)$ and $\ord_\pp(P^\beta)\ne \ord_\pp(\varepsilon_1^{N_1})$ for all $\pp\in S_{\rm fin}$. 

\noindent\textit{Proof of Claim B1:}
For each  $1\leq\ell \leq k$, we choose 
$$
\beta_\ell\in J\setminus \left(\left\{\frac{\ord_\pp(X)}{e_\pp}\colon \pp \dib P_i\right\}\cup \left\{\frac{-\ord_\pp(\varepsilon_1^{N_1})}{e_\pp}\colon \pp  \dib   P_i\right\} \right),
$$
arbitrarily. $\hfill\square$\\

We now fix $\beta$ so that the above holds. \\

\noindent\textbf{Claim B2:} There exists $i_0\in\{1,\dots,s_\infty+1\}$  such that 
$\ord_{\pp}(av^{N_1})\ne  \ord_{\pp}(r_{i_0}bu^{N_1})$ for every given $\pp\in S_\infty$.

\noindent\textit{Proof of Claim B2:}
We choose $i_0$ such that 
$$
 \ord_{\infty}(r_i)\in\{\ord_{\infty}(r_i)\colon i=1,\dots,s_\infty+1\}\setminus \left(\left\{\frac{\ord_\pp(av^{N_1}-\ord_{\pp}(bu^{N_1}))}{e_\pp}\colon \pp \dib \infty\right\} \right),
$$
arbitrarily. $\hfill\square$\\

We now fix $i_0$ so that the above holds. \\

\noindent\textbf{Claim B3:} There exists $j_0\in\{1,\dots,s_\infty+1\}$ such that
$$
\ord_{\pp}(av^{N_1})\ne\ord_{\pp}(bv^{N_1}P^\beta s_{j_0})\ {\rm and} \ 
\ord_{\pp}(r_{i_0}bu^{N_1})\ne \ord_{\pp}(bv^{N_1}P^\beta s_{j_0})
$$
\begin{center}
	and 
\end{center}

 $$\ord_\pp(av^{N_1}+r_{i_0}bu^{N_1}+bv^{N_1}P^\beta s_{j_0})\leq 0\ \ \  {\rm for\ all}\ \  \ \pp\in S_\infty.$$
\noindent\textit{Proof of Claim B3:}  Notice that for each $i=1,\dots,s_\infty$, we have 
$$
\ord_\pp(av^{N_1}+r_{i_0}bu^{N_1}+bv^{N_1}P^\beta s_{j}-av^{N_1}-r_{i_0}bu^{N_1}-bv^{N_1}P^\beta s_{j'})
=\ord_{\infty_i}(bv^{N_1}P^\beta(s_j-s_{j'}))$$
$$
 = \ord_{\infty_i}(bv^{N_1}P^\beta)+\ord_{\infty_i}(q_j-q_\ell) <0.
$$ 
It follows that $\ord_\pp(av^{N_1}+r_{i_0}bu^{N_1}+bv^{N_1}P^\beta s_{j})\leq 0$ except for at most one $j$. It follows that  there exists a $j_0\in \{1,\dots,s_\infty+1\}$  so that  $\ord_\pp(av^{N_1}+r_{i_0}bu^{N_1}+bv^{N_1}P^\beta s_{j_0})\leq 0$ for any  $i=1,\dots,s_\infty.$ This finishes the proof of Claim B3. $\hfill\square$\\

We will obtain our contradiction by showing that Equation \eqref{eqtocontradict} does not hold for this choice of $i_0$ and $j_0$.

To ease the reading, let us write $\rho=r_{i_0}$, $\sigma=P^\beta s_{j_0}$ and 
$$
T=av^{N_1}+\rho bu^{N_1}+bv^{N_1}\sigma.
$$ 

We shall obtain an upper bound for $\sum_{\pp\notin S} \ord_\pp(b^2c-a^2d)\deg_K(\pp).$ In order to do this, note that

\begin{equation*}
\begin{split}
\sum_{\pp\notin S} \ord_\pp(b^2c-a^2d)\deg_K(\pp) & \leq \\
\sum_{\substack{\ord_\pp(b^2c-a^2d)>0}} \ord_\pp(b^2c-a^2d)\deg_K(\pp) & = \\
\sum_{\substack{\ord_\pp(b^2c-a^2d)<0}} -\ord_\pp(b^2c-a^2d)\deg_K(\pp) & =
\end{split}
\end{equation*}

\begin{equation*}
\begin{split}
-\sum_{\substack{\ord_\pp(b^2c-a^2d)<0\\ \pp\in S_{\rm fin}}} \ord_\pp(b^2c-a^2d)\deg_K(\pp)-\sum_{\substack{\ord_\pp(b^2c-a^2d)<0\\ \pp\in S_{\infty}}} \ord_\pp(b^2c-a^2d)\deg_K(\pp)\\
\leq 3sD\left [h_K-\min\{\ord_{\min}(a),\ord_{\min}(b),\ord_{\min}(c),\ord_{\min}(d)\} \right ]
\end{split}
\end{equation*}
where the last inequality uses Lemma \ref{keylemma} to get the bound for the sum over the primes in $S_{\rm fin}$, and the strong triangular inequality for the sum over the primes in $S_\infty$. So we have from Equation \eqref{eqtocontradict}
\begin{multline}\label{equpperbound}
\sum_{\pp\notin S} \ord_\pp(T)\deg_K(\pp)\le\\
3 s  D \left [h_K-\min\{\ord_{\min}(a),\ord_{\min}(b),\ord_{\min}(c),\ord_{\min}(d)\} \right ].
\end{multline}


Now we shall obtain a lower bound for $\sum_{\pp\notin S} \ord_\pp(T)\deg_K(\pp)$. The lower bound is much harder to get, and it is where all the assumptions about the different parameters come into play. Observe that 

\begin{equation*}
\begin{aligned}
\sum_{\pp\notin S} \ord_\pp(T)\deg_K(\pp)  &= 
\sum_{\substack{\ord_\pp(T)>0}} \ord_\pp(T)\deg_K(\pp)  - 
\sum_{\substack{\ord_\pp(T)>0\\ \pp\in S}} \ord_\pp(T)\deg_K(\pp) \\
&= 
\sum_{\substack{\ord_\pp(T)>0}} \ord_\pp(T)\deg_K(\pp)  - 
\sum_{\substack{\ord_\pp(T)>0\\ \pp\in S_{\rm fin}}} \ord_\pp(T)\deg_K(\pp) \\
& = -\sum_{\substack{\ord_\pp(T)<0}} \ord_\pp(T)\deg_K(\pp)  
-\sum_{\substack{\ord_\pp(T)>0\\ \pp\in S_{\rm fin}}} \ord_\pp(T)\deg_K(\pp),
\end{aligned}
\end{equation*}
where the second equality comes from our hypothesis on $T$. We need a last claim. \\

\noindent\textbf{Claim C:} For each $\pp\in S_{\rm fin}$, if $\ord_\pp(T)> 0$, then 
$$
\ord_\pp(T)
\leq \max\{\ord_\pp(\sigma),\ord_\pp(b)\}\leq C_3-s D\ord_{\min}(b).
$$

\noindent\textit{Proof of Claim C:} Recall that by our choice of  $\rho$ and $\sigma$ we have that  
$$
0<\ord_\pp(T)=\min\{\ord_\pp(av^{N_1}),\ord_\pp(\rho bu^{N_1}),\ord_\pp(bv^{N_1}\sigma)\}.
$$  
We proceed by cases:

\begin{enumerate}
	\item[{\bf Case 1}]: If $\ord_\pp(a)>0,$ then $\ord_\pp(b)=0$ by our choice of $a$ and $b$. Thus, $\ord_\pp(u)>0$ and $\ord_\pp(v)=0$. It follows that $\ord_\pp(T)\leq \ord_\pp(\sigma)$. 
	\item[{\bf Case 2}]: If $\ord_\pp(a)\leq 0$, then $\ord_\pp(v)>0$. This implies $\ord_\pp(u)\leq 0$ (by our choice of $u$ and $v$), which in turns implies that $\ord_\pp(T)\leq \ord_\pp(b)$. 
	
	This proves the claim. $\hfill\square$
\end{enumerate} 
Finally, we obtain  
$$
-\sum_{\substack{\ord_\pp(T)<0}} \ord_\pp(T)\deg_K(\pp)
-\sum_{\substack{\ord_\pp(T)>0\\ \pp\in S_{\rm fin}}} \ord_\pp(T)\deg_K(\pp)
$$
$$
\geq -N_1\min\{\ord_{\min}(v),\ord_{\min}(u)\}-s D\max\{-\ord_{\min}(a),-\ord_{\min}(b)\}
$$
$$ 
-s D(C_3-s D\ord_{\min}(b)).
$$
$$
\geq -N_1\min\{\ord_{\min}(v),\ord_{\min}(u)\}+s D\min\{\ord_{\min}(a),\ord_{\min}(b)\}
$$
$$ 
+s D(-C_3+s D\ord_{\min}(b)).
$$

So, together with Equation \eqref{equpperbound}, we have
\begin{multline}\notag
-N_1\min\{\ord_{\min}(v),\ord_{\min}(u)\}-\\
s D\min\{\ord_{\min}(a),\ord_{\min}(b)\}-\\
s D(-C_3+s D\ord_{\min}(b) ) \le
\end{multline}
$$
3 s  D \left [h_K-\min\{\ord_{\min}(a),\ord_{\min}(b),\ord_{\min}(c),\ord_{\min}(d)\} \right ],
$$
hence
\begin{multline}\notag
-N_1\min\{\ord_{\min}(v),\ord_{\min}(u)\}\le\\
3 s  D \left [h_K-\min\{\ord_{\min}(a),\ord_{\min}(b),\ord_{\min}(c),\ord_{\min}(d)\} \right ]\\
-s D\min\{\ord_{\min}(a),\ord_{\min}(b)\}-\\
s D(-C_3+s D_{\max}\ord_{\min}(b) ),
\end{multline}
hence by equation \eqref{lasteq}, we have
\begin{multline}\notag
-N_1\min\{\ord_{\min}(v),\ord_{\min}(u)\}\le\\
3 s  D \left [h_K-s_{\infty}
D\min\{\ord_{\min}(u),\ord_{\min}(v)\} \right ]\\
-ss_{\infty}
D^2\min\{\ord_{\min}(u),\ord_{\min}(v)\}-\\
ss_{\infty} D^2(-C_3+s D\min\{\ord_{\min}(u),\ord_{\min}(v)\} ),
\end{multline}
hence 
$$
-N_1M_{u,v}\le
3 s  D h_K-3 s s_{\infty} D^2M_{u,v}
-ss_{\infty}
D^2M_{u,v}-ss_{\infty} D^2(-C_3+s DM_{u,v} )
$$
$$
-N_1M_{u,v}\le
3 s D h_K-s s_{\infty} D^2(4M_{u,v}+(-C_3+s DM_{u,v} )
$$


By our choice of $N_0$ and $N$, the last inequality does not hold. This finishes the proof of the Theorem.
\endproof

As a corollary we obtain the undecidability of $\rm{Th}^{+\exists }(\mathcal M_S)$.
\begin{corollary}
	The structure $\mathcal M_S$ is undecidable.
\end{corollary}
\proof It follows immediatly from the undecidability of $\O_{K,S}$ in the language of rings (see \cite{S07}). \endproof




%
%
%
%
%
%
%

\noindent Carlos A. Mart\'inez-Ranero\\
Email: cmartinezr@udec.cl\\
Homepage: www2.udec.cl/~cmartinezr\\
\noindent Javier Utreras\\
Email: javierutreras@udec.cl \hspace{10pt} javutreras@gmail.com\\

\noindent Xavier Vidaux\\
Email: xvidaux@udec.cl \hspace{10pt} vidauxx@yahoo.fr \\

\noindent Same address: \\
Universidad de Concepci\'on, Concepci\'on, Chile\\
Facultad de Ciencias F\'isicas y Matem\'aticas\\
Departamento de Matem\'atica\\

\end{document}